\documentclass[preprint,10pt]{elsarticle}

\usepackage{amsmath,amsopn,amssymb,epsfig}

\newtheorem{Theorem}{Theorem}[section]
\newtheorem{Lemma}{Lemma}[section]

\newproof{proof}{Proof}
\newproof{pot}{Proof of Theorem \ref{thm2}}

\newcommand{\bb}{\begin{bmatrix}}
\newcommand{\eb}{\end{bmatrix}}
\newcommand{\bl}[1]{\begin{list}{#1}{\usecounter{bean}}} \newcommand{\el}{\end{list}}
\newcommand{\bel}[1]{\begin{equation} \label{#1}} \newcommand{\eel}{\end{equation}}

\begin{document}

\date{}
\begin{frontmatter}
\title{
On the $\top$-Stein equation $X=AX^\top B+C$
}
\author{Matthew M. Lin \fnref{fn2}}
\ead{mlin@math.ccu.edu.tw}
\address{Department of Mathematics, National Chung Cheng University, Chia-Yi 621, Taiwan.}
\author{Chun-Yueh Chiang\corref{cor1}\fnref{fn1}}
\ead{chiang@nfu.edu.tw}
\address{Center for General Education, National Formosa
University, Huwei 632, Taiwan.}

\cortext[cor1]{Corresponding author}
\fntext[fn2]{The first
author was supported by the National Science Council of Taiwan under grant
101-2115-M-194-007-MY3.}
\fntext[fn1]{The second
author was supported  by the National Science Council of Taiwan under grant
NSC100-2115-M-150-001.}

\date{ }

\begin{abstract}
We consider the $\top$-Stein equation
$X = AX^\top B + C$, where the operator $(\cdot)^\top$ denotes 
the transpose ($\top$) of a matrix.
In the first part of this paper, we analyze necessary and sufficient conditions for the existence and uniqueness of the solution $X$.
In the second part, a numerical algorithm for solving $\top$-Stein equation is  given under the solvability conditions.
\end{abstract}


\end{frontmatter}

\section{Introduction}

The purpose of this article is to analyze the solvability conditions
of the $\top$-Stein equation
\begin{equation}\label{DTS}
X = AX^\top B+C,
\end{equation}
where $A,\, B,\, C \in\mathbb{C}^{m\times n}$ are known matrices, $X
\in \mathbb{C}^{m\times n}$ is a matrix to be determined.
Our interest in the $\top$-Stein equation originates from
the study of completely integrable mechanical systems, that is, the
analysis of the $\top$-Sylvester equation
\begin{equation}\label{CTS}
A X + X^\top B = C,
\end{equation}
where $A$, $B$, $X$ are matrices in $\mathbb{R}^{m\times n}$~\cite{Braden98, Ma2003}.
By means of the generalized inverses
or QZ decomposition~\cite{Bojanczyk1992}, the solvability conditions of~\eqref{CTS} are studies in~\cite{Braden98, Ma2003, Chiang2012}.
Suppose that $(A,B^\top)$ is regular, that is,
$ a A + b  B^\top$ is invertible for some scalars $a$ and $b$.
The $\top$-Sylvester equation~\eqref{CTS} can be written as\begin{equation}\label{add}
(aA + b B^\top)X + X^\top(aB+bA^\top) = aC+ b C^\top.
\end{equation}
Premultiplying both sides of~\eqref{add} by $(aA + b B^\top)^{-1}$, we have
\begin{equation}
X + UX^\top V = W,
\end{equation}
where $U = (aA + b B^\top)^{-1}$, $V = aB+bA^\top$ and
$ W =  (aA + b B^\top)^{-1}(aC+ b C^\top)$. This is of the form~\eqref{DTS}.
In other words, numerical approaches for solving~\eqref{CTS} can be obtained by
transforming~\eqref{CTS} into the form of~\eqref{DTS}, and then applying some iterative methods to~\eqref{DTS} for the solution~\cite{Chiang2012,Su2010, Wang2007}.
With this in mind, in this note we are interested in the study of
$\top$-Stein matrix equation~\eqref{DTS}.

Observe that if $X$ is a solution of~\eqref{DTS}, then $X$ also
satisfies the following Stein equation
\begin{equation}\label{DS}
X = AB^\top X A^\top B+C+AC^\top B.
\end{equation}
It follows that once there exists an unique solution to~\eqref{DS}, we can usually solve~\eqref{DTS} via the conventional Smith-type iterative methods given in~\cite{Smith68, Zhou2009} to~\eqref{DS}. However, even if \eqref{DTS} is uniquely solvable, it  does not imply~\eqref{DS} is uniquely solvable.
For example, if $b = 1$ and $a = -1$, the scalar equation $x=-x+c$ has an unique solution $x=\dfrac{c}{2}$, but no more information can be obtained from the Stein equation $x=x+c-c$.
%
%
In our work, we formulate the necessary and sufficient conditions
for the existence of the solution of~\eqref{DTS} directly by means
of the spectrum analysis in Section 2. We study the capability of the Smith-type iterative algorithms for solving~\eqref{DTS} in Section 3
and concluding remarks are given in Section 4.

%
%
\section{Solvability conditions of the Matrix Equation $X = AX^\top B+C$}
In order to formalize our discussion, let the notations $A\otimes
B$ be the Kronecker product of matrices $A$ and $B$, $I_n$ be the $n\times n$ identity matrix, and $
\sigma(A)$ be the set of all eigenvalues of $A$.
With the
Kronecker product, \eqref{DTS} can be written as
the enlarged linear system
\begin{align}\label{KronD}
(I_{mn}-(B^\top\otimes A)\mathcal{P}) \mbox{vec}(X)=\mbox{vec}(C),
\end{align}
where $\mbox{vec}(X)$ stacks the columns of $X$ into a column vector
and $\mathcal{P}$ is the Kronecker permutation matrix~\cite{Bernstein2009} which maps
$\mbox{vec}(X)$ into $\mbox{vec}(X^\top)$, i.e., $\mathcal{P}=
\sum\limits_{1\leq i,j\leq mn}e_je_i^\top \otimes e_ie_j^\top$,
where $e_i$ denotes the $i$-th column of the $mn\times mn$ identity
matrix $I_{mn}$. Due to the specific structure of $\mathcal{P}$, it
has been shown in~\cite[Corollary 4.3.10]{Horn1994} that
\begin{equation*}
\mathcal{P}^\top (B^\top\otimes A) \mathcal{P}  = A\otimes B^\top.
\end{equation*}
It then follows that
\begin{equation}\label{kroeig}
 ((B^\top\otimes A) \mathcal{P} )^2= (B^\top\otimes A)\mathcal{P}\mathcal{P}^\top (A\otimes B^\top)   = B^\top A\otimes A B^\top,
\end{equation}
since $\mathcal{P}^2 = I_{mn}$ and $\mathcal{P} = \mathcal{P}^\top$.
Note that $\sigma(A^\top B) = \sigma(B^\top A) = \sigma(AB^\top)$.
By~\eqref{kroeig} and the property of the Kronecker
product~\cite[Theorem 4.8]{zhang1999}, we know that
$\sigma(((B^\top\otimes A) \mathcal{P})^2) = \left \{ \lambda_i
\lambda_j | \lambda_i, \lambda_j \in \sigma(A^\top B) = \left\{
\lambda_1,\ldots,\lambda_n\right\},  1\leq i, j\leq n \right\}$.
That is,  the eigenvalues of $(B^\top\otimes A) \mathcal{P}$ is
related to the square roots of the eigenvalues of $\sigma(A^\top
B)$, but from~\eqref{kroeig}, no more information can be used
to decide the positivity or non-negativity of the eigenvalues of
$(B^\top\otimes A) \mathcal{P}$. A question immediately arises as to
whether it is possible to obtain the explicit expression of the
eigenvalues of $(B^\top\otimes A) \mathcal{P}$, provided the
eigenvalues of $A^\top B$ are given. In the following two Lemmas,
we first show that the generalized inverse of an upper triangular matrix is still upper triangular and then apply it to discuss the eigenvalues of
$(B^\top\otimes A) \mathcal{P}$.

%
%
%
%


\begin{Lemma}\label{ginverse}
Suppose that $A\in\mathbb{C}^{m\times n} $ is an upper triangular matrix having full rank. Then the generalized inverse of $A$ is an upper triangular matrix.
\end{Lemma}
\begin{proof}
Let
$A =  \bb {\widehat{A}}^\top & 0\eb^\top$, if $m\geq n$, or 
$A = \bb \widehat{A} &*\eb$, if  $m\leq n$. 
Here, $\widehat{A}$ is an upper triangular matrix
with $\mbox{rank}(\widehat{A}) =\mbox{rank}(A)$. Then,
the generalized inverse $A^\dagger$ of $A$
is an $n\times m$ matrix
satisfies
$A^\dagger =
\bb \widehat{A}^{-1} & 0\eb$, if $m\geq n$, or 
$A^\dagger = \bb \widehat{A}^{-T} & 0 \eb^\top$  if$m\leq n$.
This concludes that $A^\dagger$ is upper triangular.

\end{proof}

%
%
%
\begin{Lemma}\label{Lem1}
Let $A$ and $B$ be two matrices in $\mathbb{C}^{m\times n}$. Then
\begin{enumerate}
\item There exist unitary matrices $P\in\mathbb{C}^{m\times m}$ and $Q\in\mathbb{C}^{n\times n}$ such that
$U_A := PAQ$ and $U_B := Q^H B^\top P^H$   are two upper triangular
matrices.

\item $(B^\top\otimes A)\mathcal{P} =(Q\otimes P^H)(U_A\otimes U_B)\mathcal{P}(Q^H\otimes P)$

\item $\sigma((B^\top\otimes A)\mathcal{P})=
\left \{\lambda_i , \pm\sqrt{\lambda_i\lambda_j} |
 \lambda_i, \lambda_j \in \sigma(A^\top B)  = \left\{ \lambda_1,\ldots,\lambda_n\right\},
1\leq i <  j\leq n \right\}$.
\end{enumerate}
Here,  $\sqrt{z}$ denotes the principal square root of a complex number $z$.
\end{Lemma}

\begin{proof}
Assume first that $B$ is full rank and $P AB^\top P^H=T_1$ 
is the upper schur form of $AB^\top$. 
This defines the unitary matrix $P$ and  one can 
consider the $QR$ and $QL$ decomposition of matrices $B^\top P^H$ and $(PA)^H$, respectively, such that
\begin{align*}
\begin{array}{rcll}
Q^H B^\top P^H &=&T_2,     &\mbox{for}\, m\geq n,\\
\widehat{Q}^H (PA)^H&=&\widehat{T}_2^\top, & \mbox{for}\, m<n,
\end{array}
\end{align*}
where $Q$ and $\widehat{Q}$ are unitary, and $T_2$ and $\widehat{T}_2$ are upper triangular. By Lemma~\ref{ginverse}, there exist two generalized inverses $T_2^\dagger$ and $\widehat{T}_2^\dagger$ of matrices $T_2$ and $\widehat{T}_2$, respectively, such that $T_2 T_2^\dagger = I_n$ and $\widehat{T}_2^\dagger \widehat{T}_2 = I_m$. 
In turn, one shall have the following upper triangular matrices: 
\begin{align*}
\begin{array}{lll}
U_A := P A Q=T_1 T_2^{\dagger}, & U_B:=
Q^H B^\top P^H=T_2, & \mbox{for } m\geq n, \\
U_A:=  PA\widehat{Q} = \widehat{T}_2 , & U_B:=\widehat{Q}^H B^\top P^H  = \widehat{T}_2^{\dagger} T_1, & \mbox{for } m< n.
\end{array}
\end{align*}

If $B$ is not full rank, one then considers the SVD decomposition of $B$:
\begin{equation*}
B=U\mbox{diag}(\sigma_1,\cdots, \sigma_r,0,\cdots,0) V^\top,
\end{equation*}
where $U\in\mathbb{C}^{m\times s}$ and $V\in\mathbb{C}^{s\times n}$ are unitary,  $s=\min(m,n)$, and $r=\mbox{rank}(B)$. This then defines a sequence of nonsingular matrices  
\begin{equation*}
B_k:= B+\frac{1}{k} \sum\limits_{i=r+1}^s U_{i}V_{i}^\top, 
\end{equation*}
for any positive integer $k$. It follows from the discussion above that 
for any positive integer $k$, there exist unitary matrices $P_k\in\mathbb{C}^{m\times
m}$ and $Q_k\in\mathbb{C}^{n\times n}$ such that $U_{A_k} :=
P_kAQ_k$ and $U_{B_k} := Q_k^H B_k^\star P_k^H$ are two upper triangular matrices. Since 
$\| P_k\|=\|Q_k\|=1$, $\|U_{A_k}\|_F=\|A\|_F$, and $\|U_{B_k}\|_F=\|B_k\|_F\leq
\|B\|_F+\frac{s-r}{k}$, it follows from the Bolzano --Weierstrass Property~\cite{Rudin1976} that  there exists a convergent subsequence of $(Q_k, P_k, B_k)$ such that its limit is equal to $(Q, P, B)$
for some unitary matrices $Q$ and $P$, which completes the proof of part~1.

From part~1, there exist two unitary matrices $P$ and $Q$ giving rise to
\begin{eqnarray*}
(B^\top\otimes A)\mathcal{P} &=&(Q\otimes P^H)(U_B\otimes U_A)(P\otimes Q^H)\mathcal{P}\\
&=&(Q\otimes P^H)(U_A\otimes U_B)\mathcal{P}(Q^H\otimes P).
\end{eqnarray*}
Thus part~2 holds.

Let the diagonal entries of $U_A$ and $U_B$ be denoted by
$\{a_{ii}\}$ and $\{b_{jj}\}$, respectively. 
Then, $(U_A \otimes
U_B)$ is an upper triangular matrix with given diagonal entries, specified by $a_{ii}$ and $b_{jj}$. Here, one can assume without lose of generality that 
$U_A$ and $U_B$ are $n \times n$ matrices, i.e., $m = n$.
After multiplying $(U_A \otimes U_B)$ with $\mathcal{P}$ from the
right, the position of the entry $a_{ii}b_{jj}$ is changed to be in the $j+n(i-1)$-th row and the $i+n(j-1)$-th column of the matrix
$(U_A\otimes U_B)\mathcal{P}$. They are then reshuffled by a
sequence of permutation matrices to form a block upper triangular
matrix with diagonal entries arranged in the following order
\begin{eqnarray}\label{BigMat}
&&\left\{ a_{11}b_{11}, \bb
                          0 & a_{11}b_{22} \\
                          a_{22}b_{11} & 0
                      \eb,\ldots,
                      \bb
                          0 & a_{11}b_{nn} \\
                          a_{nn}b_{11} & 0
                      \eb,
                      a_{22}b_{22},
                      \bb
                          0 & a_{22}b_{33} \\
                          a_{33}b_{22} & 0
                      \eb,\right. \nonumber
                      \\
                      &&\left.\ldots,
                      \bb
                          0 & a_{nn}b_{22} \\
                          a_{22}b_{nn} & 0
                      \eb,\ldots,
                      \bb
                          0 & a_{n-1,n-1}b_{nn} \\
                          a_{nn}b_{n-1,n-1} & 0
                      \eb,
                      a_{nn}b_{nn}
 \right\}
\end{eqnarray}
By~\eqref{BigMat}, it can be seen that
\begin{eqnarray*}
\sigma((B^\top\otimes A)\mathcal{P}) &=& \left\{ a_{ii}b_{ii},
\pm\sqrt{a_{ii}a_{jj}b_{ii}b_{jj}}, 1\leq i, j\leq n
\right \} \\
& = &
\left \{\lambda_i , \pm\sqrt{\lambda_i\lambda_j},
1\leq i, j\leq n
\right\}
\end{eqnarray*}
where $\lambda_i = a_{ii}b_{ii} \in \sigma(A^\top B)$  for $1\leq i
\leq n$. 
Note that if $m\neq n$, the diagonal entries given in~\eqref{BigMat} will remain unchanged, except that some zero diagonal entries will be added or removed from the diagonal.

\end{proof}

Note that when $m =n$, the decomposition given in part~1 of Lemma~\ref{Lem1} is called \emph{the periodic Schur decomposition}~\cite{Bojanczyk1992}.
The following result, providing the unique solvability conditions of
\eqref{DTS},
 is an immediate consequence of Lemma~\ref{Lem1}.
%
\begin{Theorem}\label{DTSEXIST}
The $\top$-Stein matrix equation~\eqref{DTS} is
uniquely solvable if and only if the following conditions are
satisfied: for $\lambda\in\sigma(A^\top B)$ and $\lambda\neq -1$,
$\dfrac{1}{\lambda}\not\in\sigma(A^\top B)$;Ê
 $-1$ can be an eigenvalue of the matrix $A^\top B$, but must be simple.
\end{Theorem}

\begin{proof}
From~\eqref{KronD}, we know that the $\top$-Stein matrix
equation~\eqref{DTS} is  uniquely solvable if and
only if
\begin{equation}\label{SolCond}
1\not\in\sigma((B^\top\otimes A) \mathcal{P}).
\end{equation}
By Lemma~\ref{Lem1},
%
%
if $\lambda \in \sigma(A^\top B)$, then $\dfrac{1}{\lambda}
\not\in\sigma(A^\top B)$. 
Otherwise, $1 = \sqrt{\lambda \cdot
\dfrac{1}{\lambda}} \in ((B^\top\otimes A) \mathcal{P})$.
On the other hand, if $-1\in
\sigma(A^\top B)$ and $-1$ is not a simple eigenvalue, then $1\in
\sigma((B^\top\otimes A) \mathcal{P})$. This
verifies~\eqref{SolCond} and the proof of the theorem is complete.
\end{proof}

\section{Smith-type iterative methods}

Originally, Smith-type iterative methods~\cite{Zhou2009} are developed to solve the Stein equation
\begin{equation*}\label{Stein}
X = AXB +C,\quad A\in\mathbb{R}^{m\times m},\, B\in\mathbb{R}^{n\times n},\, C\in\mathbb{R}^{m\times n}.
\end{equation*}
Our main thrust in this section is to extend the Smith-type iterative methods to solve~\eqref{CTS}
by means of the formula~\eqref{DS} and also preserve the convergence properties embedded in the original methods.
Here we summarize the iterative methods as follows.

 \begin{itemize}

\item{\bf The $r$-Smith iteration methods}:\begin{equation*}
X_{k+1} = \sum\limits_{i=0}^{r-1} A_k^i X_k B_k^i ,
\end{equation*}
where
\begin{eqnarray*}
X_0  =  AC^\top B + C, \quad
A_0 &=& AB^\top,\quad A_{k+1} = A_k^r, \quad k\geq 0 ,\\
B_0 &=& A^\top B,\quad B_{k+1} = B_k^r, \quad k\geq 0.
\end{eqnarray*}

\end{itemize}

It can be seen that
if $\rho(A^\top B)=\rho(A B^\top)<1$, the $r$-Smith iteration converges to the unique solution of~\eqref{DTS}
\begin{equation*}
X=\sum\limits_{i=0}^\infty (AB^\top)^i (AC^\top B+ C) (A^\top B)^i.
\end{equation*}
When $r = 2$, the $r$-Smith iteration is also called {\bf the Smith accelerative iteration}~\cite{Smith68} and it has been shown in~\cite{Zhou2009} that the Smith accelerative iteration performs more effective than any other  $r$-Smith iterations.
One possible drawback of the Smith-type iterative
methods is that it cannot handle the case when the eigenvalue $-1$ of 
$A^\top B$ is simple.
Based on the solvable conditions given in this work, it is
possible to develop a specific technique working on the particular
case and it is a subject currently under investigation.


\section{Conclusion}
In this paper, we obtain necessary and sufficient conditions for the existence and uniqueness of the solution of the $\top$-Stein equation and find that the solvability conditions between~\eqref{DTS} and~\eqref{DS} are different. 
When the equation~\eqref{DTS} is consistent, we extend the capacity of Smith-type iteration methods from the Stein equation to the $\top$-Stein equation.

%

%
\bibliographystyle{elsarticle-num}


\end{document}